\documentclass[12pt]{amsart}

\usepackage{amsmath}
\usepackage{amsfonts}
\usepackage{amssymb}
\usepackage{amsthm}
\usepackage{hyperref}
\usepackage{enumerate}
\usepackage{cleveref}
\usepackage{mathrsfs}
\usepackage[top=0.9in, bottom=1.15in, left=1.1in, right=1.1in]{geometry}
 \usepackage{hyperref}
\hypersetup{colorlinks=true,linkcolor=blue,citecolor=magenta}
\newtheorem{theorem}{Theorem}

\newtheorem{proposition}[theorem]{Proposition}
\newtheorem{corollary}[theorem]{Corollary}
\newtheorem{definition}[theorem]{Definition}
\newtheorem*{remark}{Remark}
\newtheorem*{example*}{Example}

\Crefname{conjecture}{Conjecture}{Conjectures}

\theoremstyle{plain}

\theoremstyle{plain}

\author{Philip Cuthbertson and Robert Schneider}

\address{Department of Mathematical Sciences\newline
Michigan Technological University\newline
Houghton, Michigan 49931, U.S.A.}
\email{pecuthbe@mtu.edu}

\address{Department of Mathematical Sciences\newline
Michigan Technological University\newline
Houghton, Michigan 49931, U.S.A.}
\email{robertsc@mtu.edu}

\title[Multimodal sequences and their generating functions]{Multimodal sequences and their\\ generating functions}
 
\begin{document}
 
\begin{abstract}
We define integer multimodal sequences, which are generalizations of unimodal sequences having multiple local peaks of equal size. The generating functions for multimodal sequences represent novel types of $q$-series that combine generating functions for both integer partitions and  integer compositions. We prove a bijection between multimodal sequences of equal size (sum), and show that multimodal generating functions become finite series at roots of unity like the ``strange'' function of Kontsevich, quantum modular forms, and other examples of this phenomenon in the $q$-series literature.   

\end{abstract}

\maketitle


\section{Introduction: concepts and notations}

\subsection{Introduction}

A multimodal distribution in statistics is a probability distribution with two or more modes (local peaks), by contrast with a unimodal distribution that has one mode, and suggests multiple populations or trends  within the data \cite{stat2, stat1}.  Here we define  {\it integer  multimodal sequences},    natural   generalizations of   unimodal sequences    having multiple   peaks.
\footnote{This study grew out of the authors' unpublished preprint \cite{mixture} about   ``integer mixtures'',  a broad class of integer sequences studied privately by George Beck that  extends unimodal sequences and multicolored partitions in great generality and   includes multimodal sequences. We thank   Beck for suggesting the study of integer mixtures and for sharing      computational examples   (Private communication, 2022 -- 2023).} The generating functions for multimodal sequences represent novel types of $q$-series that combine generating functions for both integer  partitions   and   compositions.

\subsection{Partitions, compositions and unimodal sequences}
Let $\mathbb Z^+$ denote the positive integers. Let $\mathcal P$ denote the set of {\it integer partitions} \cite{Andrews}, finite sequences of positive integers notated in weakly decreasing order; e.g. $(3, 2, 2, 1), (4, 4, 1), (2)\in\mathcal P$. 
Let $\emptyset\in \mathcal P$ denote the empty partition. Let $\lambda=(\lambda_1,\lambda_2, \lambda_3, \dots,$ $\lambda_r),\   \lambda_1\geq\lambda_2\geq\dots\geq\lambda_r\geq 1$, denote a nonempty partition. Let $\ell(\lambda):=r$ denote the partition {\it length} (number of parts). Let $m_i=m_i(\lambda)$ denote the {\it multiplicity} (frequency) of $i$ as a part of the partition. Let $|\lambda|:=\lambda_1+\lambda_2+\dots+\lambda_r$ denote the partition {\it size} (sum of parts). Let $\operatorname{lg}(\lambda)$ denote the {\it largest part} of $\lambda$. Define $\ell(\emptyset)=|\emptyset|=m_i(\emptyset)=\operatorname{lg}(\emptyset)=0$. 
Furthermore, let $\mathcal C$ denote the set of {\it integer compositions} \cite{MacMahon}, finite {\it ordered} sequences of positive integers;  e.g. $(3, 2, 1, 2), (4,4,1),$ $(4, 1, 4)\in\mathcal C$. Let $\emptyset\in \mathcal C$ denote the empty composition.  
Let $\mathcal{U}$ denote  {\it unimodal sequences} \cite{Murty, Stanley}, finite sequences of positive integers that increase from an initial part size (weakly or strictly), attain a peak value with multiplicity $k\geq 1$ (a ``$k$-fold peak''),  then decrease (weakly or strictly) until terminating, possibly with left or right  subsequence empty. {We extend the  partition notations introduced above (length, size, etc.) to   compositions, unimodal sequences and multimodal sequences, with the same meanings and terminology}. 

Let $\mathcal U_k\subset \mathcal U$ denote unimodal sequences having $k$-fold peak. E.g., here is   a unimodal sequence where  we mark the two-fold peak with overlines:
$$(2,3,3,5,\overline{7},\overline{7},6,4,1,1,1)\in\mathcal{U}_2.$$ 
A unimodal sequence with $k$-fold peak can be viewed as a partition written in   increasing order, followed by some number of peak values,   followed by a second partition    in decreasing order. The  minimal  unimodal sequence  with $k$-fold peak is $(\overline{1}, \overline{1}, \dots, \overline{1})$ with $k$ repetitions. 

\subsection{Multimodal sequences}
Let $\mathcal{M}$ denote the set of {\it multimodal sequences}, sequences of positive integers  that {increase} (weakly or strictly) {until they reach a peak value, then oscillate sporadically between lesser values, reach the same peak again, oscillate sporadically, reach the peak, oscillate again, and so on until $k\geq 1$ peaks are attained, then decrease} (weakly or strictly) {until terminating}. We say a multimodal sequence has $k$-fold peak if there are $k$ occurrences of the peak value, in any arrangement.   
We note   $\mathcal U\subset \mathcal{M}$. 

Let $\mathcal{M}_k\subset \mathcal M$ denote  multimodal sequences having $k$-fold peak, $k\geq 1$. E.g., here is a multimodal sequence where we overline the three-fold peak to highlight it:  
\begin{equation}\label{mmseq} (1, 3, 5, \overline{7}, 1, 4,  3,   \overline{7}, 4, 6, 2, 1, 3, \overline{7}, 6,4,2,2, 1)\in \mathcal{M}_3.\end{equation}

A multimodal sequence can be viewed as the concatenation of a left-hand and a right-hand partition with a sequence of intermediate compositions, all separated by peak values.   We note that any of the extremal partitions or intermediate compositions might be empty. Thus a unimodal sequence having $k$-fold peak represents the special case of  a   sequence in $\mathcal M_k$ where {all} the    intermediate compositions are empty, and we have that $\mathcal U_k\subset \mathcal M_k$. The minimal multimodal sequence with $k$-fold peak is also $(\overline{1}, \overline{1}, \dots, \overline{1})$ with $k$ repetitions.

Let $\mathcal{M^{**}}$  denote the set of {\it strongly multimodal sequences} in which the left- and right-hand partitions are  {\it strictly} increasing and decreasing, respectively. Let $\mathcal M^{*}$ denote   {\it semi-strongly multimodal sequences}, for which we require     the left-hand partition to be strictly increasing (i.e., having distinct parts), while the right-most partition is  unrestricted.  
 
\begin{remark}
    Multimodal sequences are    similar to   the multicompositions studied in \cite{HopkinsOuvry}, in that both types of sequences arise from concatenations of integer compositions. 
\end{remark}

\section{Two-variable generating functions}
\subsection{Generating functions for partitions and unimodal sequences}

For  $z,q\in\mathbb C$, recall the $q$-{\it Pochhammer} symbol \cite{Andrews}, defined by the product
\begin{equation}\label{qPoch}
(z;q)_{n}:=\prod_{k=0}^{n-1}(1-zq^k).\end{equation}
As $n\to \infty,$ set  $(z;q)_{\infty}:=\lim_{n\to \infty}(z;q)_{n}$ for $|q|<1$.  

A standard method in partition theory is the use of product-sum generating functions, such as the following two-variable version of Euler's partition generating function product formula (see \cite{Berndt}). 
For $|zq|<1$, we have 
\begin{equation}\label{gen1} 
\frac{1}{(zq;q)_{n}}\  =\  \sum_{\substack{\lambda\in\mathcal P\\ \operatorname{lg}(\lambda)\leq n}} z^{\ell(\lambda)}q^{|\lambda|},\end{equation}
and letting $n\to \infty$, 
\begin{equation}\label{gen2}
\frac{1}{(zq;q)_{\infty}}\  =\  \sum_{\lambda\in \mathcal P}z^{\ell(\lambda)}q^{|\lambda|},\end{equation}
 with the first sum taken over the set of partitions with largest part at most $n$, and the second sum taken over  all  partitions. Under certain substitutions and   with multiplicative factors, expression \eqref{gen2} is well known to be a modular form (see e.g. \cite{Ono_web}). 

Now, note that $zq^n(zq;q)_{n}^{-1}$ is the generating function for partitions $\lambda$ with largest part exactly equal to $n\geq 1$ (weighted by $z^{\ell(\lambda)}$); the $(zq;q)_{n}$ factor in the denominator generates partitions with parts at most $n$, while the $zq^n$ factor in the numerator generates a largest part of size $n$. Then summing over $n \geq 1$ generates all nonempty partitions, and one   has  
\begin{equation}\label{formula2}
1+ \sum_{n\geq 1}\frac{zq^n}{(zq;q)_n}\  =\  \sum_{\lambda\in \mathcal P}z^{\ell(\lambda)}q^{|\lambda|}.
\end{equation}

Analogously, note for  $k\geq 1, n\geq 1,$ that $z^kq^{nk}(zq;q)_{n-1}^{-2}$ generates the set of unimodal sequences having $k$-fold peak of size $n$; one factor $(zq;q)_{n-1}$ in the denominator generates the set of increasing, left-hand partitions with parts at most $n-1$ and the second factor $(zq;q)_{n-1}$ generates the decreasing right-hand partitions with parts at most $n-1$, with the $z$ parameter keeping track of their lengths, while $z^kq^{kn}$ generates $k$ copies of the peak value $n$. Then the generating function for {\it unimodal sequences $\mathcal{U}_k$  having $k$-fold peak}  is 
\begin{equation}
 \sum_{n\geq 1}\frac{z^k q^{kn}}{(zq;q)_{n-1}^2} \  
=\  \sum_{u\in \mathcal{U}_k}z^{\ell(u)}q^{|u|}.
\end{equation}
The generating function for unrestricted unimodal sequences involves a double sum:
\begin{equation}\label{double}
\sum_{k\geq 1} \sum_{n\geq 1}\frac{z^k q^{kn}}{(zq;q)_{n-1}^2} \  
=\  \sum_{u\in \mathcal{U}}z^{\ell(u)}q^{|u|}.
\end{equation}

Likewise, the generating function for {\it strongly} {unimodal sequences $\mathcal{U}_k^{**}$  having $k$-fold peak}, where  the left and right sides are strictly increasing and decreasing, respectively, is 
\begin{equation}
 \sum_{n\geq 1}{z^k q^{kn}}{(-zq;q)_{n-1}^2} \  
=\  \sum_{u\in \mathcal{U}_k^{**}}z^{\ell(u)}q^{|u|}.
\end{equation}
Just as in \eqref{double} above, a double sum will generate  all  strongly unimodal sequences.

Note that unimodal  generating functions are built from partition generating functions.

\subsection{Composition-theoretic $q$-Pochhammer symbol}

We define a bracket notation  that will serve as something of a composition-theoretic analogue of the $q$-Pochhammer symbol \eqref{qPoch}. This will help us to see similarities, as well as highlight dissimilarities, between partition and composition generating functions.

\begin{remark} This subsection draws from combinatorial and analytic ideas proved in \cite{SS1, SS2}. \end{remark}
\begin{definition} For $z,q\in \mathbb C$, set $\{z;q\}_0:=1$ and, for $n\geq 1$,   let \begin{equation*}
\{z;q\}_n:=1-z \sum_{k=0}^{n-1}q^k.\end{equation*}
For $|q|<1$, as $n\to \infty$    let \begin{equation*}
\{z;q\}_{\infty}:=\lim_{n\to \infty}\{z;q\}_n. \end{equation*}
\end{definition}


If $q=1$, then $\{z;1\}_n=1-nz$. For $q\neq 1$,  use the following  identities from geometric series:
\begin{equation}\label{compsymbol}
\{z;q\}_n=\frac{1-q-z+zq^{n}}{1-q}, \  \  \  \  \  \  \  \  \  \  \  \   \{z;q\}_{\infty}=\frac{1-q-z}{1-q},\end{equation}
with the restriction $|q|<1$ in the second case. We note   $\{z;q\}_n=0$ if and only if $q$ is a root of   the polynomial $zx^n-x+(1-z)\in\mathbb C[x]$, and $\{z;q\}_{\infty}=0$ if and only if $q=1-z$.

Unlike the product $(z;q)_{n}$, this   $ \{z;q\}_{n}$ is a sum, so we cannot expect   usual $q$-series analytic techniques \cite{Andrews, Berndt, Fine} to hold. On the other hand, 
  the   sets $\mathcal P$ and $\mathcal C$ enjoy   many combinatorial similarities. In some cases, {\it analytic} similarities also carry over;\footnote{Dualities between partitions and compositions is a motivating  theme of \cite{Keith-S-S, SS1, SS2}.} one can write down   composition generating functions  analogous to equations  \eqref{gen1} and \eqref{gen2}.

\begin{proposition}\label{prop} For $z,q\in \mathbb C,\  |q|<\frac{1}{1+|z|}$, we have 
$$\frac{1}{\{zq;q\}_{n}}\  =\  \sum_{\substack{c\in\mathcal C\\ \operatorname{lg}(c)\leq n}} z^{\ell(c)}q^{|c|},$$
 with the   sum taken over  compositions with largest part at most $n$. Furthermore,  $$\frac{1}{\{zq;q\}_{\infty}}\  =\  \sum_{c\in\mathcal C} z^{\ell(c)}q^{|c|},$$ with the   sum taken over  all  compositions. \end{proposition}


\begin{proof} 
The first identity is the case $S=\{1, 2, 3, \dots, n\}$, and the second identity is the case $S= \mathbb Z^+$,  of Proposition 1.1 in \cite{SS2} which states for $S\subseteq  \mathbb Z^+$, $|q|<\frac{1}{1+|z|}$,  
we have \begin{equation}\label{prop1.5}
\frac{1}{1-z\sum_{n\in S}q^n}\  =\  \sum_{c\in\mathcal C_S} z^{\ell(c)}q^{|c|},\end{equation}
with $\mathcal C_S$ denoting compositions into parts from $S$.   The convergence condition on $q$ is sufficient  for the identities to hold for all $S\subseteq \mathbb Z^+$; for a given subset $S$, the necessary condition is $\left|z\sum_{n\in S}q^n\right|<1$. We refer the reader to the proof of \eqref{prop1.5} in \cite{SS2}, which is based on the multinomial theorem and geometric series.
\end{proof}

At this stage, one might anticipate further generating function analogies between partitions and compositions, since both represent sums of natural numbers.  Is there also, for instance,  a composition $q$-series identity  
$$1+\sum_{n\geq 1}\frac{zq^n}{\{zq;q\}_n}\   \stackrel{?}{=}\  \sum_{c\in \mathcal C}z^{\ell(c)}q^{|c|},$$
 analogous to \eqref{formula2}? It is not quite this simple. Ordered versus unordered parts 
 is a critical combinatorial distinction in proving composition-theoretic $q$-series generating functions.
\begin{proposition} For $z,q\in \mathbb C$ such that $|q|<\frac{1}{1+|z|}$, and writing compositions in the form  $c=(c_1, c_2, c_3, \dots, c_r)\in\mathcal C,\   c_i\geq 1$, we have 
$$1+\sum_{n\geq 1}\frac{zq^n}{\{zq;q\}_n} = \sum_{\substack{c\in\mathcal C\\ c_1 =\operatorname{lg(c)}}}z^{\ell(c)}q^{|c|},$$
with the right-hand sum taken over compositions having largest part size in position $c_1$. \end{proposition}

\begin{proof}
Consider each summand  $zq^n\{zq;q\}_n^{-1}$ on the left. The denominator generates compositions with parts at most $n$, while $zq^n$ inserts a part of maximal size $n$ into the first (left-most) position; thus $c_1=\operatorname{lg}(c)$ for each composition $c$ in the right-hand sum. 
\end{proof}

We note another generating function analogy that does not hold:  while $(-zq;q)_{n}$ generates partitions into distinct parts at most $n$, $\{-zq;q\}_{n}$ generates compositions into   a single  part at most $n$. See \cite{distinct} for   more on generating functions for compositions into distinct parts. 

\subsection{Generating functions for multimodal sequences}

Recall   that a multimodal sequence can be viewed as a   partition on the left with parts written in   increasing order, followed by a peak value, followed by a   composition into smaller parts than the peak, then another peak value, then another smaller composition, and so on until the final peak value has been attained, then ending with a partition in   decreasing order on the right.

   The generating functions for multimodal sequences form a complex class of $q$-series. 
\begin{theorem}\label{thm1}
For $z,q\in \mathbb C$ such that $|q|<\frac{1}{1+|z|}$, $k\geq 1$, the generating function $M_k(z;q)$ for multimodal sequences $\mathcal M_k$ with $k$-fold peak  is given by
$$M_k(z;q)= \sum_{n\geq 1}\frac{z^k q^{kn}}{\{zq;q\}_{n-1}^{k-1}(zq;q)_{n-1}^2}=\sum_{\mu\in\mathcal{M}_k}z^{\ell(\mu)}q^{|\mu|}.$$
The generating function $M(z;q)$ for unrestricted multimodal sequences $\mathcal M$ is given by
$$M(z;q)=\sum_{k\geq 1}\sum_{n\geq 1}\frac{z^k q^{kn}}{\{zq;q\}_{n-1}^{k-1}(zq;q)_{n-1}^2}=\sum_{\mu\in\mathcal{M}}z^{\ell(\mu)}q^{|\mu|}.$$
 \end{theorem}

\begin{proof} We combine aspects of the preceding generating function proofs.  We note   that $M_k(z;q)$ is defined   on $|zq|<  1$ except, for $k\geq 2$, deleting from the domain points $q$   that are roots of  $P_{z,n}(x):=zx^{n+1}-(z+1)x+1=(1-x)\{zx;x\}_n$ for any $n\geq 1$; but   does not      have the interpretation as a multimodal   generating function unless $|q|<\frac{1}{1+|z|}$,  which is  usually a sufficient   convergence condition on composition-theoretic series of the shape $\sum_{c\in \mathcal C}$    \cite{SS1, SS2}, and excludes    the  polynomial roots   that give rise to poles of order $k-1$.  

In the $n$th summand of  $M_k(z;q)$, one $(zq;q)_{n-1}$ factor in the  denominator generates the {partitions} with parts $<n$ to the left of the multimodal sequence, and  a second  $(zq;q)_{n-1}$ factor  generates {partitions} with parts $<n$ to the right, including $\emptyset\in \mathcal P$. The $\{zq;q\}_{n-1}^{k-1}$ factor  generates the  $k-1$ compositions into parts $<n$ that lie between the  peaks, including $\emptyset\in \mathcal C$; while $z^kq^{kn}$ generates the $k$-fold peaks of size $n$. For the second identity,   
sum both sides over $k\geq 1$ to generate the entire set $\mathcal M$. 
\end{proof}

Recall that we define  a {\it strongly} multimodal sequence by requiring the left-hand partition to be strictly increasing and the right-hand partition to be strictly decreasing;  whereas a {\it semi-strongly} multimodal sequence requires   the left-hand partition to be strictly increasing but places no restriction on the right-hand partition.

\begin{theorem}\label{strongthm}
For $z,q\in \mathbb C$ such that $|q|<\frac{1}{1+|z|}$,    
the generating function $M^{**}(z;q)$ for   strongly multimodal sequences $\mathcal M^{**}$ is given by 
$$M^{**}(z;q)=\sum_{k\geq 1}\sum_{n\geq 1}\frac{z^k q^{kn}(-zq;q)_{n-1}^2}{\{zq;q\}_{n-1}^{k-1}}=\sum_{\mu\in\mathcal{M}^{**}}z^{\ell(\mu)}q^{|\mu|}.$$
The generating function $M^{*}(z;q)$ for   semi-strongly multimodal sequences $\mathcal M^{*}$ is given by
$$M^{*}(z;q)=\sum_{k\geq 1}\sum_{n\geq 1}\frac{z^k q^{kn}(-zq;q)_{n-1}}{\{zq;q\}_{n-1}^{k-1}(zq;q)_{n-1}}=\sum_{\mu\in\mathcal{M}^{*}}z^{\ell(\mu)}q^{|\mu|}.$$
 \end{theorem} 

\begin{proof}
    To prove these identities,  proceed exactly as in  the proof of the second identity in Theorem \ref{thm1}, but in the present cases any factor $(-zq; q)_{n-1}$ in the numerator generates partitions into distinct parts less than $n$. Then for $|q|<\frac{1}{1+|z|}$, we define 
    \begin{flalign} \label{multi1}
       M_k^{**}(z;q)\  =\   \sum_{n\geq 1}\frac{z^kq^{kn}(-zq;q)_{n-1}^2}{\{zq;q\}_{n-1}^{k-1}},\  \  \  \  \  
       M_k^{*}(z;q)\  =\  \sum_{n\geq 1}\frac{z^kq^{kn}(-zq;q)_{n-1}}{\{zq;q\}_{n-1}^{k-1}(zq;q)_{n-1}},
    \end{flalign}
which are two-variable generating functions for strongly unimodal sequences with $k$-fold peak as well as semi-strongly multimodal sequences with $k$-fold peak, respectively. 
Summing these two  series over   $k\geq 1$ generates  the sets $\mathcal M^{**}$ and $\mathcal M^*$, respectively. 
\end{proof}

Along similar lines, one may concatenate partitions and compositions with diverse restrictions, and combine their generating functions, to derive further classes of multimodal sequences and their generating functions. For further reading on partition generating functions, see \cite{Andrews, Berndt, Fine}. For more on composition generating functions, see \cite{Keith-S-S, SS1, SS2}.

\subsection{Additional  composition-theoretic $q$-Pochhammer formulas}

We record a handful of miscellaneous product and summation formulas related to $\{zq;q\}_{n}$ and $\{zq;q\}_{\infty}$. 

 \begin{theorem}\label{productthm}
 For $z,q\in \mathbb C$, if $q$ is not a root of the polynomial $P_{z,n}(x)=zx^{n+1}-(z+1)x+1\in \mathbb C[x]$ for any $n\geq 1$, we have 
 \begin{flalign*} \{zq;q\}_n\  =\  \prod_{i=1}^{n}\bigg(1-\frac{zq^i}{\{zq;q\}_{i-1}}\bigg),\  \  \  \  \  \  \  \  
 \frac{1}{\{zq;q\}_n}\  =\  \prod_{i=1}^{n}\bigg(1+\frac{zq^i}{\{zq;q\}_{i}}\bigg).\end{flalign*}
If $|q|<\frac{1}{1+|z|}$, then  both identities hold  when $n\to \infty$.
\end{theorem}

\begin{proof}
Noting   $\{zq;q\}_{i-1}-zq^i=\{zq;q\}_{i}$ and $\{zq;q\}_{0}=1$,  then the first equation in Theorem \ref{productthm} can be interpreted as a telescoping product with all intermediate terms canceling:
\begin{equation}
\{zq;q\}_{n}=\frac{\{zq;q\}_{n}}{\{zq;q\}_{n-1}}\cdot \frac{\{zq;q\}_{n-1}}{\{zq;q\}_{n-2}}\cdot \frac{\{zq;q\}_{n-2}}{\{zq;q\}_{n-3}}  \cdots \frac{\{zq;q\}_{2}}{\{zq;q\}_{1}}\cdot \frac{\{zq;q\}_{1}}{\{zq;q\}_{0}}.
\end{equation}
In the same manner, noting  $\{zq;q\}_{i}+zq^i=\{zq;q\}_{i-1}$ leads to  the second equation.
\end{proof}

It is a consequence of these product formulas, filtered through standard partition generating function techniques, that the composition-theoretic symbol $\{zq;q\}_{\infty}^{\pm 1}$ can be expressed in terms of  $q$-series summed over integer partitions.

 \begin{corollary}\label{productcor2}
Let $\mathcal P^*\subset \mathcal P$ denote the set of {\it partitions into distinct parts} (no repetition). For $z,q\in \mathbb C$ such that $|q|<\frac{1}{1+|z|}$, we have
 \begin{flalign*}\{zq;q\}_{\infty}\  &=\  \sum_{\lambda \in \mathcal P^*}\frac{(-z)^{\ell(\lambda)}q^{|\lambda|}}{\prod_{i\in\lambda}\{zq;q\}_{i-1}},\\
 \frac{1}{\{zq;q\}_{\infty}}\  &=\  \sum_{\lambda \in \mathcal P^*}\frac{z^{\ell(\lambda)}q^{|\lambda|}}{\prod_{i\in\lambda}\{zq;q\}_{i}},
  \end{flalign*}
as well as 
$$\frac{1}{\{zq;q\}_{\infty}}\  =\  \sum_{\lambda \in \mathcal P}\frac{z^{\ell(\lambda)}q^{|\lambda|}}{\prod_{i\in\lambda}\{zq;q\}_{i-1}},
$$
with the sums taken over partitions into distinct parts in the first two identities, over unrestricted partitions in the third, and  $\prod_{i\in\lambda}$  taken over the parts $i\in \mathbb Z^+$ of partition $\lambda$.
\end{corollary}

\begin{proof}
Let $n\to \infty$ in Theorem \ref{productthm} when $|q|<(1+|z|)^{-1}$, recalling this  restriction on $q$ avoids the polynomial roots specified in Theorem \ref{productthm}, and apply the identity $\prod_{j\geq 1}(1+a_j q^j) = \sum_{\lambda \in \mathcal P^*}q^{|\lambda|}\prod_{i \in \lambda}a_i$ (see e.g. \cite{Schneider_zeta}) to the right-hand products, to give the first two equations. Apply  $\prod_{j\geq 1}(1-b_j q^j)^{-1} = \sum_{\lambda \in \mathcal P}q^{|\lambda|}\prod_{i \in \lambda}b_i$ (e.g.  \cite{Schneider_zeta}) in the reciprocal of the initial product for $\{zq;q\}_{\infty}$ in Theorem \ref{productthm}, for the third equation.
\end{proof}

\section{Bijections between multimodal sequences of fixed size}

\subsection{Multimodal bijections} As we noted previously, the composition-theoretic symbol $\{z;q\}_n$ is a less convenient analytic object than $(z;q)_{n}$; it is not subject to many of the   factoring and cancellation methods that give rise to $q$-series generating function identities. 

However,  multimodal sequences are also partially composed of partitions and these are rich with set-theoretic bijections. For instance, it is well known that the number of partitions of size $N\geq 1$ into distinct parts is equal to the number of partitions of size $N$ into odd parts. Semi-strongly multimodal sequences enjoy a similar bijection.


For $x\in\mathbb R$, let $\lfloor x \rfloor$ denote the  usual {\it floor function}.

\begin{theorem}\label{bijection}
    For $N\geq 1$, the number of semi-strongly multimodal sequences of size $N$ is equal to the number of multimodal sequences of size $N$ with the following properties:
    \begin{enumerate}[(i)]
        
        \item given  peak(s) of size $n$, the left-hand partition consists entirely of odd parts $<n$; 

        \item given  peak(s) of size $n$, the right-hand partition has unrestricted parts of size $\leq \lfloor n/2 \rfloor$ while parts   greater than $\lfloor n/2 \rfloor$ appear with multiplicity exactly equal to one. 
    \end{enumerate} 
\end{theorem}

\begin{proof}
   Let $m^*(N)$ denote the number of semi-strongly multimodal sequences of size equal to $N\geq 1$, and let $m^{\#}(N)$ denote the number of multimodal sequences of size $N$ with left-hand partition having property ($i$) in the statement of the theorem and right-hand partition having property ($ii$). Setting $z=1$ in  Theorem \ref{strongthm},  then for $|q|<1/2$ the generating function for semi-strongly multimodal sequences can be rearranged to give 
\begin{flalign}
\sum_{N\geq 1}m^*(N)q^N \  &=\  \sum_{k\geq 1}M_k^*(1;q)=\  \sum_{k\geq 1}\sum_{n\geq 1}\frac{ q^{kn}(-q;q)_{n-1}}{\{q;q\}_{n-1}^{k-1}(q;q)_{n-1}}\\
\nonumber &=\  \sum_{k\geq 1}\sum_{n\geq 1}\frac{ q^{kn}\prod_{i\leq \lfloor \frac{n-1}{2}\rfloor}(1+q^i)  \prod_{ \lfloor \frac{n-1}{2}\rfloor<i\leq n-1}(1+q^i)}{\{q;q\}_{n-1}^{k-1} \prod_{\substack{j\leq n-1\\ \text{$j$ is odd}}}(1-q^j) \prod_{i\leq \lfloor \frac{n-1}{2}\rfloor}(1-q^{2i})}\\
\nonumber &=\  \sum_{k\geq 1}\sum_{n\geq 1}\frac{ q^{kn}  \prod_{ \lfloor \frac{n-1}{2}\rfloor<i\leq n-1}(1+q^i)}{\{q;q\}_{n-1}^{k-1} \prod_{\substack{j\leq n-1 \\ \text{$j$ is odd}}}(1-q^j) \prod_{i\leq \lfloor \frac{n-1}{2}\rfloor}(1-q^{i})}\\
\nonumber &=\  \sum_{N\geq 1}m^{\#}(N)q^N.
\end{flalign}
 Comparing coefficients in the left-  and right-hand power series proves the theorem. \end{proof}

\begin{remark}
    For $k=1$, note $M_1^{*}(1;q)$ is the generating function for unimodal sequences having one-fold peak and with   left-hand subsequence   composed of distinct parts; these are called  ``semi-strictly convex sequences'' in \cite{Andrews2}. Thus semi-strictly convex sequences   of size $N$ are in bijection with the set of one-fold unimodal sequences such that     (i) and (ii) apply. 
\end{remark}

\section{Finite $q$-series formulas at roots of unity}\label{gen}

\subsection{Unimodal rank generating functions at roots of unity} Multimodal sequences are natural generalizations of unimodal sequences. The {\it rank} of a unimodal sequence is defined (see \cite{BOPR}) as the number of terms (parts) to the right of the peak minus the number of terms to the left.  
For $|q|<1$, $z$ a certain root of unity, the $q$-series generating functions for unimodal ranks enjoy    interesting analytic behaviors as $q$ approaches   roots of unity. 

Let $\mathcal U^{**}$ denote the set of {\it strongly unimodal sequences with single peak} (i.e., one-fold peak) such that the left- and right-hand subsequences are strictly increasing and strictly decreasing, respectively, and let $\operatorname{rk}(u)$ denote the rank of  $u\in\mathcal U$  a unimodal sequence. The  rank generating function for strongly unimodal sequences with single peak   \cite{BOPR},  
\begin{equation}\label{unigen}
U(z;q)=\sum_{n=1}^{\infty}q^{n} (-z^{-1}q;q)_{n-1} (-zq;q)_{n-1} =\sum_{u\in \mathcal U^{**}}z^{\operatorname{rk}(u)}q^{|u|}
\end{equation}
with $|q|<1$, is connected  to  mock modular forms, partial theta functions, quantum modular forms and other functions in the universe of modular forms theory \cite{BFR, BOPR,  Lovejoy, Schneider_jtp}.  
 
In \cite{BOPR}, it is proved that $U(\pm i; q)$ is a third order mock theta function and $U(-1; q)$ is a quantum modular form. A key property used in \cite{BOPR} is that $U(-1;q)$ becomes finite at   roots of unity in a particular sense, much like the ``strange'' function of Kontsevich (see \cite{Zagier}),  viz. if we set $q=\zeta_{m}=e^{2 \pi\operatorname{i}/m}$ then $1-q^m=1-\zeta_{m}^m=0$, thus $(\zeta_m;\zeta_m)_{M}=0$ for all $M\geq m$ to yield (after a change of indices) a truncated series 
\begin{equation}\label{unigen2}
U(-1;\zeta_m):=    \lim_{q\to \zeta_m} U(-1; q)\  =\  \sum_{n=0}^{m-1} \zeta_m^{n+1} (\zeta_m; \zeta_m)_{n}^2,
\end{equation}
where   $q\to \zeta_m$ radially from inside the unit disk; and  likewise as $q$ approaches $\zeta_m^j,\  j\in\mathbb Z$. Analogous identities hold for strongly unimodal sequences with $k$-fold peak. 
 
 It is counterintuitive that a series like $U(-1; q)$ becomes both finite in magnitude and in summation length as $q$ approaches the boundary of its domain of convergence {\it at discrete points on the unit circle}. Generically we expect  $q$-series to diverge at the unit circle. 
 
\subsection{Multimodal generating functions at roots of unity} This  kind of ``finiteness at roots of unity'' behavior has been  proved  in the literature to extend to  quantum modular forms, ``strange'' functions, certain mock theta functions, and other $q$-series; see e.g. \cite{Bringmann, BFR, Rolen, Finite, Schneider_jtp}. The generating functions for multimodal sequences represent a special class of $q$-series, mixing  together generating functions for integer partitions as well as compositions.  How do $q$-series involving the composition-theoretic generating function $\{zq;q\}_n^{-1}$ behave at roots of unity? Do they enjoy similar finiteness behavior? 

Throughout this subsection of the paper, we consider $M_k(z;q), M_k^{**}(z;q)$ and $M_k^{*}(z;q)$ as    complex-valued functions in two complex variables rather than as generating functions. Then the convergence condition $|zq|< 1$ suffices as we do not need  combinatorial interpretations, so long as  $q$ avoids the roots of $zx^{n+1}-(z+1)x+1,\  n\geq 1,$   for $k\geq 2$. 
Moreover, setting $\zeta_m=e^{2\pi \operatorname{i}/m}$,  
if $A(q)$ is a function valid for $|q|<1$ as $q\to \zeta_m^j, \  j\in\mathbb Z,$     we define 
\begin{equation}
    A(\zeta_m^j):=\lim_{q\to \zeta_m^j} A(q)
\end{equation}
if the limit exists, where $q\to \zeta_m^j$ radially from within the unit disk, as in   \eqref{unigen2} above. 

It turns out that specializations of multimodal generating functions  enjoy finiteness behavior similar to   $U(-1;q)$, quantum modular forms, and   other functions    noted above. 

\begin{theorem} Let $\zeta_m:=e^{2\pi \operatorname{i}/m},\  m\in \mathbb Z^+$, and for $k\geq 1$ define $M_k^{**}(z;q)$ and $M_k^{*}(z;q)$ as in equation \eqref{multi1}  but with the      conditions  that $|zq|<  1$ and that if $k\geq 2$ then $q$ is not a root of the polynomial $P_{z,n}(x)=zx^{n+1}-(z+1)x+1\in\mathbb C[x]$ for any $n\geq 1$. Set $\zeta:=\zeta_m^j$ for $j\in\mathbb Z$ such that there exists a radial line segment $(1-t)\zeta, \  t\in (0,  \varepsilon),$ that is free of roots of $P_{z,n}(x)$ for some $\epsilon>0$. If $m$ is an even value, then we have the finite evaluations 
 \begin{flalign*}
     M_k^{**}(1;\zeta)&\  =\   \sum_{n=1}^{m/2}\frac{\zeta^{kn}(-\zeta;\zeta)_{n-1}^2}{\{\zeta;\zeta\}_{n-1}^{k-1}},\\
     M_k^{*}(1;\zeta)&\  =\   \sum_{n=1}^{m/2}\frac{\zeta^{kn}(-\zeta;\zeta)_{n-1}}{\{\zeta;\zeta\}_{n-1}^{k-1}(\zeta;\zeta)_{n-1}}.
     \end{flalign*}
\end{theorem}

\begin{proof}
Set $\zeta=\zeta_m^j$. If   $m$ is even,  then $\zeta^{m/2}=e^{\pi\operatorname{i}}=-1$ and $(-\zeta, \zeta)_{n-1}=0$  in the numerator of every term $n>m/2$, giving both identities. There is one fine point to the second identity: there are poles arising from the $(\zeta;\zeta)_{n-1}$ factor in the denominators of the summands for all $n > m$; in each   term,  these are canceled by a zero from the $(-\zeta;\zeta)_{n-1}$  factor  in the numerator   of {twice} the order of the pole from the denominator,  as $q\to \zeta=\zeta_m^j$ radially. The stated condition that an open radial line segment $(1-t)\zeta, \  t\in (0,  \varepsilon),$ exists in a neighborhood of $\zeta$ ensures that $q$ can approach $\zeta$ along a radial path.  
\end{proof}


The preceding limit formulas are finite because of a   vanishing factor in the numerator; some $q$-series   enjoy finiteness at roots of unity   without  this kind of  vanishing factor, for less obvious reasons. In Lemma 9 of \cite{Schneider_jtp}, the second author of this paper gives a quite general formula showing finiteness type phenomena for broad classes of infinite series.  Below, we prove an extension of \cite[Lemma 9]{Schneider_jtp} in order to give  a finite formula for $M_k (z; q)$. 

\begin{theorem}\label{lemma}
Suppose $\phi\colon \mathbb Z^+ \to \mathbb C$ is a periodic function of period $m\in\mathbb Z^+,\  m\neq 1$, i.e., $\phi(mj+r)=\phi(r)$ for all $j\in \mathbb Z$. Define $f\colon \mathbb Z^+ \to \mathbb C$ by 
the product
\begin{equation*}
f(n):=\prod_{i=1}^{n}\phi(i).
\end{equation*}
Let $g\colon \mathbb Z^+ \to \mathbb C$ be a periodic function of the same period $m\neq 1$. Define the summatory function $F_g(n)$ by $F_g(0):=0$ and, for $n\geq 1$,
\begin{equation*}
F_g(n):=\sum_{j=1}^{n}f(j)g(j),\  \  \  \  \  \    F_g(\infty):=\lim_{n\to \infty}F_g(n)\  \text{if the limit exists}.
\end{equation*}
Then the following formulas hold true:
\begin{enumerate}[(i)]
\item \label{periodic3} For $0\leq r <m$ we have
\begin{equation*}
F_g(mj+r)=\frac{1-f(m)^j}{1-f(m)}F_g(m)\  +\  f(m)^j F_g(r).
\end{equation*}
\item \label{periodic4}
For $|f(m)|<1$ we have  
\begin{equation*}
F_g(\infty)=\frac{F_g(m)}{1-f(m)}.
\end{equation*}
\end{enumerate}
\end{theorem}  
  
\begin{proof} 
We proceed similarly to the proof of \cite[Lemma 9]{Schneider_jtp}. We will require $m>1$ or else the proof is trivial. Begin by noting that since $\prod_{i=1}^m \phi(i)=\prod_{i=mt+1}^{m(t+1)}\phi(i),\  t\geq 0$, then
\begin{equation}\label{periodic5}
f(mj)=\prod_{i=1}^{mj}\phi(i)=\left(\prod_{i=1}^{m}\phi(i)\right)^j=f(m)^j,
\end{equation}
by the periodicity of $\phi$. Then by the definition of $F_g(n)$   together with (\ref{periodic5}),   we  rewrite
\begin{flalign}
&F_g(mj+r)\\ \nonumber&=\sum_{i=1}^{m}f(i)g(i)+\sum_{i=m+1}^{2m}f(i)g(i)+\sum_{i=2m+1}^{3m}f(i)g(i)+\dots\\ \nonumber &    \  \  \ \  \  \  \  \  \  \  \  \  \  \  \  \  \  \  \  \  \  \  \  \  \  \  \  \  \  \  \  \  \  \  \  \  \  \  \  \  \  \  \  \  +\sum_{i=m(j-1)+1}^{mj}f(i)g(i)+\sum_{i=mj+1}^{mj+r}f(i)g(i)\\
\nonumber &=\left(1+f(m)+f(m)^2+...+f(m)^{j-1}\right)\sum_{i=1}^{m}f(i)g(i)+f(m)^j\sum_{i=1}^{r}f(i)g(i).
\end{flalign}
Recognizing  that $1+f(m)+f(m)^2+...+f(m)^{j-1}$ is a finite geometric series proves  (\ref{periodic3}); if $|f(m)|<1$,      as $j\to\infty$ the sum is a convergent infinite geometric series, proving (\ref{periodic4}).
\end{proof}

\begin{remark}
    Lemma 9 of \cite{Schneider_jtp} represents the   $g\equiv 1$ case of Theorem \ref{lemma}.
\end{remark}



By Theorem \ref{lemma}, we can write down a finite formula for $M_k(z;q)$, the generating function  for multimodal sequences with $k$-fold peak, at certain roots of unity.

\begin{corollary}\label{cor} Let $\zeta_m:=e^{2\pi \operatorname{i}/m},\  m>1$, and for $k\geq 1,\  z,q\in \mathbb C,$ define $M_k(z;q)$ as in Theorem \ref{thm1} but with the    conditions that $|q|<  1,$ that if $k\geq 2$ then $q$ is not a root of the polynomial $P_{z,n}(x)=zx^{n+1}-(z+1)x+1\in\mathbb C[x]$ for any $n\geq 1$, that $|z|=1$,\  $z\neq -1$,  and that $|z^m-1|>1$. Set $\zeta:=\zeta_m^j$ for $j\in\mathbb Z$ such that there exists a radial line segment $(1-t)\zeta, \  t\in (0,  \varepsilon),$ that is free of roots of $P_{z,n}(x)$ for some $\epsilon>0$; then we have 
  $$M_k(z;\zeta)\  =\     (z\zeta)^k\  +\  \frac{z^{k-m}(z^m-1)^2}{z^m-2}\sum_{n=1}^{m}\frac{\zeta^{k(n+1)}}{\{z\zeta;\zeta\}_{n}^{k-1}(z\zeta;\zeta)_{n}^2},$$
if the limit exists as $q\to\zeta$ radially for a given choice of $z$. \end{corollary}

\begin{remark}
Translating the conditions on $z$ into a radial   setting using trigonometry, set $z=e^{i\theta}$ such that  $\pi/3<\theta <5\pi/3$ and  $\theta\neq \pi$, with $z$ not an $m$th root of unity.
\end{remark}

\begin{proof}
We re-index the terms of $M_k(z;q)$ in order to conveniently apply  Theorem \ref{lemma}: 
\begin{equation}
    M_k(z;q):= \sum_{n\geq 1}\frac{z^kq^{kn}}{\{zq;q\}_{n-1}^{k-1}(zq;q)_{n-1}^2}\  =\  (zq)^k+\sum_{n\geq 1}\frac{(zq)^k q^{kn}}{\{zq;q\}_{n}^{k-1}(zq;q)_{n}^2},
\end{equation}
noting the leading $(zq)^k$ term in the right-hand expression will not enjoy   periodic behavior in parameter $n$ such as the theorem exploits. 
Then set $\zeta=\zeta_m^j$ under the stated condition on $j$, and make the following substitutions into Theorem \ref{lemma}:   Set $\phi(i)=\zeta^k (1-z\zeta^{i})^{-2}$, which is periodic in the parameter $i$ with period equal to $m$, as required by the theorem; thus $f(n)=\prod_{1\leq i\leq n}\phi(i)=\zeta^{kn}(z\zeta; \zeta)_{n}^{-2}.$ Set \begin{equation}\label{eq}
    g(n) =\frac{(z\zeta)^k}{\{z\zeta;\zeta\}_{n}^{k-1}}= (z\zeta)^k \left(\frac{1-\zeta}{1- (z+1)\zeta+z\zeta^{n+1}}\right)^{k-1}
\end{equation} using identity  \eqref{compsymbol}. Thus $g(n)$ is   periodic in parameter $n$ with period   equal to $m\in\mathbb Z^+$ except in the case $m=1$; this   can be seen from the definition $\{zq;q\}_n=1-z\sum_{1\leq i\leq n} q^i$ since $q^{mr}+q^{mr+1}+q^{mr+2}+\dots+q^{m(r+1)-1}=0$ if $q=\zeta_m^j$ for all $m>1,\  r\geq 0$. But in the $m=1$ case,  $\{z;1\}_n=1-nz$ which is not periodic in $n$, hence we see the restriction $m>1$ is necessary; in any event, $m\neq 1$ is an assumption in the proof of Theorem \ref{lemma}. 

The factor $(z\zeta;\zeta)_n^{-1}$ is singular  for every root of unity $\zeta=\zeta_m^j$  if $z=1$, $n\geq m$, and for even order roots of unity $\zeta$ if $z=-1$, $n\geq m/2$. In addition, if $z=-1$, one can see that $\{z\zeta; \zeta\}_n^{-1}$ is singular at every root of unity $\zeta=\zeta_m^j$, $m>1$, when $n\equiv -1\  (\operatorname{mod} m)$. 
Hence we require $z\neq \pm 1$.  Now the corollary follows     from equation (ii) of Theorem \ref{lemma},  using the fact   $(X;\zeta_m^j)_m=1-X^m$, thus $f(m)=\zeta^{km}(z\zeta;\zeta)_m^{-2}=(1-z^m)^{-2}$. Then   $|f(m)|<1$ under the hypothesis $|z^m-1|>1$,    satisfying the convergence condition for (ii); in particular, note  that this implies $z$ cannot equal   any $m$th root of unity. However, we also need to satisfy the convergence condition  $|zq|<  1$; this is compatible with the preceding convergence conditions if we make the additional  requirement $|z|=1$.   With a little trigonometry one sees these conditions on $z$ are equivalent to setting $z=e^{i\theta}$ such that  $\pi/3<\theta <5\pi/3$,  \  $\theta\neq \pi$, with $z$  not an $m$th root of unity, as we remarked above.

Then a little algebra using Theorem \ref{lemma}  yields the corollary, assuming the limit as $q\to\zeta=\zeta_m^j$ exists for given choices of $z, m, j$. 
In \cite[Section 6]{Watson}, Watson investigates    limits of $q$-series as $q$ approaches roots of unity, and gives analytic methods for proving if limits like this exist. The reader is   pointed to the proof of Example 10 in \cite{Schneider_jtp}, which proves a    somewhat similar limit formula involving $q$-Pochhammer symbols by following Watson's steps closely. In our current case we also require an open radial line segment  to be free of roots of $P_{z,n}(x)=zx^{n+1}-(z+1)x+1,\  n\geq 1,$ as $q\to\zeta$ when $k\geq 2$ to ensure $q$   can approach $\zeta=\zeta_m^j$ via a radial path, which depends on the choices of $z, m$ and $j$. 
\end{proof}

\begin{remark}
    An interesting implication of   Theorem \ref{lemma} is that  any 
{$q$-series  enjoying finiteness at an $m$th root of unity based on  \cite[Lemma 9]{Schneider_jtp}  will  also  enjoy this property when a periodic factor of period $m\geq 2$ such as $\{z\zeta_m^j; \zeta_m^j\}_{n}^{-1}$} is inserted in the summands, so long as appropriate analytic conditions are   satisfied. 
\end{remark}

\section{Further questions}
This paper studies $q$-series mixing partition-theoretic and  composition-theoretic components.  As we noted previously,   partition generating functions and unimodal   generating functions are well known to connect to modular forms, mock theta functions, quantum modular forms, and other functions having remarkable analytic properties. Composition-theoretic generating functions also make  connections to modular forms (see e.g. \cite{Keith-S-S, SS2}). 

Recalling that multimodal sequences generalize unimodal sequences, how far do the analytic analogies   between unimodal and multimodal  generating functions extend? Does the unimodal rank statistic   extend to multimodal sequences in a meaningful way? Do multimodal sequences and their generating functions make connections to modular forms  theory comparable to  the connections made by   unimodal rank generating functions?

\section*{Acknowledgments}
We are thankful to George Beck for inspiring this research project; 
as well as George E. Andrews,   
Brian Hopkins, 
William J. Keith,  Ken Ono,  
C\'{e}cile Piret,   Larry Rolen, 
Andrew V. Sills, and Hunter Waldron for conversations that informed this work.  In particular, we are grateful to Andrew V. Sills for   combinatorial advice as well as noting corrections.   

\end{document}